\documentclass{jams-l}

\usepackage{amssymb}
\usepackage{amsthm}
\usepackage[colorlinks,citecolor=red,pagebackref,hypertexnames=false]{hyperref}
\usepackage{color}
\usepackage{esint}
\usepackage{bbm}
\usepackage{graphicx}
\usepackage{etoolbox}
\usepackage{mathrsfs}
\usepackage{enumitem}

\def\dT{\mathcal{T}}

\def\cW{{\mathscr{W}}}

\DeclareMathOperator{\diam}{diam}
\def\Lip{\mathop\mathrm{Lip}} 						
\def\dist{\mathop\mathrm{dist}} 						

\newcommand{\ps}[1]{\left( #1 \right)}

\def\XXint#1#2#3{{\setbox0=\hbox{$#1{#2#3}{\int}$ }
\vcenter{\hbox{$#2#3$ }}\kern-.58\wd0}}

\newtheorem{theorem}{Theorem}[section]
\newtheorem{lemma}[theorem]{Lemma}

\theoremstyle{definition}
\newtheorem{definition}[theorem]{Definition}

\theoremstyle{remark}
\newtheorem{remark}[theorem]{Remark}

\numberwithin{equation}{section}

\newcommand{\R}{\mathbb{R}}
\newcommand{\N}{\mathbb{N}}
\newcommand{\Z}{\mathbb{Z}}
\newcommand{\C}{\mathbb{C}}

\newcommand{\lip}{\mathrm{Lip}}
\newcommand{\hd}{\mathcal{H}^d}

\newcommand{\hdc}{\mathcal{H}^d_\infty}

\newcommand{\B}{\mathbb{B}}

\newcommand{\grass}{G(d, d+1)}

\newcommand{\betae}[1]{\beta_E^{#1, d}}
\newcommand{\xjk}{x^k_j}

\newcommand{\xik}{x^k_i}
\newcommand{\xim}{x_{i}^m}
\newcommand{\ujk}{u^{k}_{j}}
\newcommand{\uik}{u_{i}^k}
\newcommand{\uim}{u^m_{i}}
\newcommand{\Ljk}{L_{j,k}}

\newcommand{\Lim}{L_{i,m}}
\newcommand{\Pjk}{P_{j,k}}
\newcommand{\Pik}{P_{i,k}}
\newcommand{\Pim}{P_{i,m}}
\newcommand{\Bjk}{B_{j,k}}
\newcommand{\Bim}{B_{i,m}}
\newcommand{\Qjk}{Q_{j,k}}
\newcommand{\Qim}{Q_{i,m}}

\newcommand{\dt}{\frac{dt}{t}}
\newcommand{\dr}{\frac{dr}{r}}

\newcommand{\cubes}{\mathcal{D}}

\newcommand{\chara}{\mathbbm{1}}

\newcommand{\wt}{\widetilde}

\newcommand\blfootnote[1]{%
  \begingroup
  \renewcommand\thefootnote{}\footnote{#1}%
  \addtocounter{footnote}{-1}%
  \endgroup
}

\numberwithin{equation}{section}
\theoremstyle{plain}
\newtheorem{sublemma}[theorem]{Sublemma}
\newtheorem{corollary}[theorem]{Corollary}
\newtheorem{proposition}[theorem]{Proposition}

\title[Tangent points and $\beta$ numbers]{Tangent points of lower content $d$-regular sets and $\beta$ numbers}
\author{Michele Villa}

\newcommand{\Addresses}{{
  \bigskip
  \footnotesize

  M. Villa, \textsc{Maxwell Institute for Mathematical Sciences,
School of Mathematics,
University of Edinburgh,
Edinburgh,
UK,
EH9 3FD}\par\nopagebreak
  \textit{E-mail address}: \texttt{m.villa-2@sms.ed.ac.uk}

}}

\begin{document}

\maketitle

\begin{center}

\begin{minipage}[c][][r]{300pt}
\begin{small}
\textsc{Abstract.} Given a lower content $d$-regular set in $\R^n$, we prove that the subset of points in $E$ where a certain Dini-type condition on the so-called Jones $\beta$ numbers holds coincides with the set of tangent points of $E$, up to a set of $\hd$-measure zero. The main point of our result is that $\hd|_E$ is not $\sigma$-finite; because of this, we use a certain variant of the $\beta$ coefficient, firstly introduced by Azzam and Schul in [AS1], which is given in terms of integration with respect to the Hausdorff content.
\end{small}
\end{minipage}
\end{center}
\blfootnote{\textup{2010} \textit{Mathematics Subject Classification}: \textup{28A75}, \textup{28A12} \textup{28A78}.


M. Villa was supported by The Maxwell Institute Graduate School in Analysis and its
Applications, a Centre for Doctoral Training funded by the UK Engineering and Physical
Sciences Research Council (grant EP/L016508/01), the Scottish Funding Council, Heriot-Watt
University and the University of Edinburgh.
}
\tableofcontents
\section{Introduction}
Let $E \subset \R^n$. One says that $E$ is $d$-rectifiable if there are Lipschitz maps $f_i: \R^d \to \R^n$, $i=1,2,...$ so that 
\begin{align}
    \hd\ps{E \setminus \bigcup_{i=1}^\infty f_i(\R^d)} = 0. 
\end{align}
Rectifiability is a central notion in geometric measure theory. It was first introduced (albeit under a different terminology) by Besicovitch  ([Bes1]) with $d=1$ and $n=2$; the theory was then extended to general $d,n$ by Federer ([Fed]); we refer the reader to [M1] for more details. 
A classic characterisation of rectifiability is given in terms of \textit{approximate tangents}. Let $x \in E$, and let $V$ be a $d-$dimensional affine plane in $\R^n$. One says that $V$ is an approximate tangent plane for $E$ at $x$ if $\limsup_{r \to 0} \frac{\hd(B(x,r)\cap E)}{r^d} >0$ and for all $0<s<1$, we have
\begin{align} \label{eq:approx_tan}
    \lim_{r \to 0} \frac{\hd\ps{E \cap B(x,r) \setminus X(x, V, s)}}{r^d}=0,
\end{align}
where $X(x, V, s) = \{y \in \R^n\, |\, d(y-x, V) < s|y-x|\}$ (and $B(x,r)$ denotes the open ball centered at $x$ and with radius $r$). Then the following is true. 
\begin{theorem}[{See [M1], Theorem 15.19}] \label{theorem:approxtan}
Let $E$ be an $\hd$-measureable subset of $\R^n$ with $\hd(E)<\infty$. Then the following are equivalent.
\begin{itemize}
    \item $E$ is $d$-rectifiable.
    \item For $\hd$-almost every point $x \in E$, there is a unique approximate tangent $d$-plane for $E$ at x.
\end{itemize}
\end{theorem}
In other words, a set is $d$-rectifiable if and only if, at least at small scales, it lays close to a $d$-plane; in Theorem \ref{theorem:approxtan}, `laying close' is made precise by \eqref{eq:approx_tan} | one, however, can easily think of other conditions that translate `laying close' into a mathematical quantity, and consequently define a different notion of tangent points and planes. Let us give a few examples. 
 \begin{itemize}
     \item  For $x \in E$, we say that an affine $d$-plane $V$ is a $C$-tangent\footnote{$C$ here is for \textit{cone}} plane for $E$ at $x$, if for all $0<s<1$, there exists a $r_0>0$ so that for all $r<r_0$, 
\begin{align} \label{eq:tangent_cone}
    E \cap B(x,r) \setminus X(x, V, s) = \emptyset.
\end{align}
A $C$-tangent point is defined in the obvious way.
\item If $U, V \subset \R^n$, we set
\begin{align}
    d_H(U, V) := \max \left\{ \sup_{u \in U} \dist(u, V), \sup_{v \in V} \dist(v, U) \right\}.
\end{align}
This is just the Hausdorff distance between two sets.
We say that $x \in E$ is an $H$-tangent\footnote{$H$ is for Hausdorff} point if there exists a $d$-dimensional affine plane $L$ so that 
    \begin{align} \label{eq:tangent_hausdorff}
   & \lim_{r \to 0} \frac{d_H\ps{E \cap B(x, r), L \cap B(x,r)}}{r} = 0.   \end{align} 
   We will call $L$ an $H$-tangent plane.
\item A point $x \in E$ is called a tangent\footnote{We don't put any prefix here: this definition will be the most used one in the paper.} point of $E$ if there exists a $d$-dimensional affine plane so that
\begin{align} \label{eq:tangent}
    \lim_{r \to 0} \sup_{y\in B(x,r)\cap E}\frac{\dist(y, L)}{r} = 0 .
\end{align}
In this case, we will call $L$ a tangent plane.
\end{itemize}

It is immediate to see that if a point is a $C$-tangent point, then it is an approximate tangent point.
With this in mind, let us state the following theorem of Bishop and Jones
\begin{theorem}[{[BJ1], Theorem 2}] \label{tangent2d}
Let $E \subset \C$ be a Jordan curve. Except for a set of zero $\mathcal{H}^1$-measure, $x\in E$ is a $C$-tangent point of $E$ if and only if
\begin{align} \label{dini}
    \int_0^{1} \beta_{E,\infty}(x,t)^2 \, \frac{dt}{t} < \infty
\end{align}
\end{theorem}
The quantities $\beta_{E, \infty}$ appearing in $\eqref{dini}$ are the so-called Jones $\beta$-numbers | they where introduced by Jones in [J1], where he gave a solution to the Analyst's travelling salesman problem (see [R] for a survey), and they are defined as follows:
\begin{align}
    \beta_{E, \infty}^d (x,t) := \inf_{L} \sup_{y \in E \cap B(x,t)} \frac{\dist(y, L)}{r},
\end{align}
where the infimum is taken over all affine $d$-planes and $\dist(y, L)=\inf_{q \in L} |y-q|$. One may think of
$\beta_{E,\infty}(x,r)\cdot r$ as the width of the smallest strip which contains $E
\cap B(x,t)$. Theorem \ref{tangent2d} then says that that at $\mathcal{H}^1$-almost every point $x \in E_0:=\{x \in E| \int_0^{\diam(E)} \beta_{E, \infty}(x,r)^2 \dr <\infty\}$ there exists an approximate tangent plane for $E$; in particular $E_0$ is $1$-rectifiable.
Recently, Azzam and Schul introduced in [AS1] a variant of the Jones $\beta$ number, defined as follows.
\begin{definition} \label{def:betas}
Let $E \subset \R^n$ and define
\begin{align*}
    \beta_E^{d,p} (x,r) = \inf_L \left( \frac{1}{r^d} \int_0^1 \hdc (\{y \in B(x,r)\cap E \,|\, \dist(y,L)> tr\}) t^{p-1} \, dt \right)^{\frac{1}{p}},
\end{align*}
where the infimum is taken over all affine $d$-planes $L$ in $\R^n$, and $\hdc$ denotes the Hausdorff content.
\end{definition}
Azzam and Schul used this version of $\beta$ number to prove an analyst's travelling salesman theorem for lower content $d$-regular sets in $\R^n$ | their result is analogous to the one given by Jones in the plane (see Remark \ref{remark:novel} for some motivation on the usage of these coefficients). Let us recall what a lower content $d$-regular set is. 
A set $E \subset \R^n$ is \textit{$d$-lower content regular} with constant $c_0$ if for all $x \in E$ and all $\diam(E)> r>0$, it holds that
\begin{align} \label{eq:lowercontent}
    \hdc (E \cap B(x,r)) \geq c_0 r^d.
\end{align}
Examples of $d$-lower content regular sets are Reifenberg flat sets or sets satisfying Condition B (see [AS1], Remark 1.11 and references therein for definitions and more examples); a Jordan curve in $\C$ is $1$-lower content regular.  

The main result of this paper is, in a sense, a generalisation to any $d$ and $n$ of Theorem \ref{tangent2d}.

\begin{theorem} \label{theorem:main}
Let $E \subset \R^n$ be a bounded and lower content $d$-regular subset and let $E_0 \subset E$ be a subset with positive and finite $d$-dimensional Hausdorff measure. If $d=1$ or $d=2$ let $1 \leq p < \infty$; if $d \geq 3$, let $1 \leq p \leq \frac{2d}{d-2}$. Then the following are equivalent.
\begin{enumerate}[label=\textbf{T\arabic*}]
    \item For $\hd$-almost every $x \in E_0$, 
    \begin{align} \label{eq:main_beta}
    \int_0^{\diam(E)} \beta_E^{p, d}(x,t)^2  \dt < \infty. 
    \end{align}
    \label{item:main_beta}
      \item For $\hd$-almost all $x\in E_0$, there exists a $d$-dimensional affine plane $L$ so that 
   \begin{align} \label{eq:main_tan}
   \lim_{r \to 0} \sup_{y\in B(x,r)\cap E}\frac{\dist(y, L)}{r} = 0 .
\end{align}
\label{item:main_tan}
    \item For $\hd$-almost all $x\in E_0$, there exists a $d$-dimensional affine plane $L$ so that 
    \begin{align} \label{eq:main_hausdorff}
   & \lim_{r \to 0} \frac{d_H\ps{E \cap B(x, r), L \cap B(x,r)}}{r} = 0.   \end{align} 
   \label{item:main_hausdorff}
   \item For $\hd$-almost all $x\in E_0$, there exists a $d$-dimensional affine plane $L$ so that, for all $0<s<1$,
   \begin{align} \label{eq:main_approx}
       \lim_{r\to 0} \frac{\hd\ps{E \cap B(x, r) \setminus X(x, V, s) }}{r^d} = 0.
   \end{align} 
   \label{item:main_approx}
\end{enumerate}
In particular, if any of \eqref{item:main_beta}, \eqref{item:main_tan}, \eqref{item:main_hausdorff} or \eqref{item:main_approx} holds, then $E_0$ is $d$-rectifiable.
\end{theorem}

Let us make some remarks.
\begin{remark}
The content of the theorem is really in the equivalence between \ref{item:main_beta} and \ref{item:main_tan}. The equivalence between the various definition of tangents is known; we will include proofs for completeness. We chose to prove the equivalence $\ref{item:main_beta} \iff \ref{item:main_tan}$ for convenience: the definition of tangent in \eqref{eq:main_tan} seems the simplest and easiest to use.
\end{remark}

\begin{remark}
The conclusion about the rectifiability of $E_0$ in Theorem \ref{theorem:main} can be reached through different methods. For example, see Theorem 15.19 in [M1] on the equivalence of  \eqref{item:main_approx} and rectifiability. Also, that the Dini condition \eqref{item:main_beta} implies rectifiability is implicit in our proof of Proposition \ref{theorem:summabletangent}.
\end{remark}

\begin{remark}[On related works]
The relationship between rectifiability and a Dini-type condition on (some type of) Jones $\beta$ numbers has been an active area of research in the past thirty years. Starting with the already mentioned work of Jones [J1], these coefficients have been applied and used to understand geometric aspects of sets and measure in a variety of context; see for example David and Semmes in [DS1] and [DS2] (the reader will find here very interesting connections between harmonic analysis and rectifiability); David and Toro in [DT1] (see also Ghinassi in [G1]), and Edelen, Naber and Valtorta in [ENV1] (these papers relates the $\beta$ coefficients with parameterisation of Reifenberg flat sets); Schul in [Sch1] for a generalisation of the traveling salesman theorem of Jones on curves to Hilbert spaces; the series of results by Tolsa and Azzam and Tolsa [T1], [T2] and [AT1] for a characterisation of rectifiability in terms of a Dini square integrability condition on the $\beta$ coefficients; the series of results by Badger and Schul [BS1], [BS2], [BS3], for characterisations of $1$-rectifiable measure in terms of $\beta$ coefficients without density assumptions.
\end{remark}

\begin{remark} \label{remark:novel}
Let us mention why is the theorem novel. In any situation where one is dealing with sets of dimension larger than one, the $\beta_{E, \infty}$ coefficient become rather useless (at least when we want to consider a Dini-type condition on them, see the example in [Fa]). David and Semmes then introduced an averaged version of such numbers, given as
\begin{align}
    \beta_{E,p}^d(x,r):= \inf_{L} \ps{\frac{1}{r^d} \int_{B(x,r)} \ps{\frac{\dist(y, L)}{r}}^p \, d \hd|_{E}(y) }^{\frac{1}{p}}.
\end{align}
Note that for $\beta_{E,p}^d$ to even make sense, $\hd|_E$ needs to be at least $\sigma$-finite. However, a generic lower content $d$-regular set will have dimension \textit{strictly larger than $d$} | just as in the case of Theorem \ref{tangent2d}, where Bishop and Jones considered a Jordan curve. For this reason we decided to use the $\beta$ coefficient introduced by Azzam and Schul: the integral there is a Choquet integral with respect to the Hausdorff content: it makes sense in any case. 
\end{remark}

\begin{remark}[On the range of $p$'s]
The range of $p$'s for which Theorem \ref{theorem:main} is true should not be (so) surprising. The results of David and Semmes for uniformly rectifiable sets and the Traveling Salesman Theorem of Azzam and Schul, for example, hold for the very same range. 
The `so' in brackets is there to say that things are not as not-surprising as they seem. The characterisation of $d$-rectifiable sets in terms of $\beta$ numbers by X. Tolsa in one direction (see [To1]) and by Azzam and Tolsa in the other (see [AT1]), holds only for $p=2$. The question remained open on whether this result could be extended to the 'expected' range of $p$: Tolsa's recent examples show that such a characterisation \textit{can only hold} for $p=2$ (see [To2]).
\end{remark}

\subsection{Acknowledgement}
I would like to thank Jonas Azzam, my supervisor; his help and patience have been indispensable to make any progress in this work.

\section{Preliminaries}

\subsection{Notation} \label{sec:notation}
We gather here some notation and some results which will be used later on.
We write $a \lesssim b$ if there exists a constant $C$ such that $a \leq Cb$. By $a \sim b$ we mean $a \lesssim b \lesssim a$.
In general, we will use $n\in \N$ to denote the dimension of the ambient space $\R^n$, while we will use $d \in \N$, with $d\leq n-1$, to denote the dimension of a subset $E \subset \R^n$, in the sense of lower content regularity.

For two subsets $A,B \subset \R^n$, we let
$
    \dist(A,B) := \inf_{a\in A, b \in B} |a-b|.
$
For a point $x \in \R^n$ and a subset $A \subset \R^n$, 
$
    \dist(x, A):= \dist(\{x\}, A)= \inf_{a\in A} |x-a|.
$
We write 
$
    B(x, t) := \{y \in \R^n \, |\,|x-y|<t\},
$
and, for $\lambda >0$,
$
    \lambda B(x,t):= B(x, \lambda t).
$
At times, we may write $\B$ to denote $B(0,1)$. When necessary we write $B^n(x,t)$ to distinguish a ball in $\R^n$ from one in $\R^d$, which we may denote by $B^d(x, t)$. 
See refer the reader to [Ma1] for more detail on Hausdorff measures and content. 
For $0<p<\infty$ and $A \subset \R^n$ Borel, we define the $p$-Choquet integral as 
\begin{align*}
    \int_A f(x)^p\, d \hdc(x) := \int_0^\infty \hdc(\{x \in A\, |\, f(x)>t\}) \, t^{p-1}\, dt.
\end{align*}
See [AS1], Section 2 and Appendix for more on Choquet integration.

\subsection{Coherent families of balls and planes}

A main tool which we will use is the theorem by G. David and T. Toro on parameterisations of Reifenberg flat sets with holes (see [DT1]). To state more precisely the result given there, we need to introduce a few more notions and definitions. 

\begin{definition}
Set the normalised local Hausdorff distance between two sets $E, F$ to be given by
\begin{align}
    d_{x,r}(E,F) := \frac{1}{r} \max \left\{ \sup_{y \in E\cap B(x,r)} \dist(y, F) \,; \, \sup_{y \in F \cap B(x,r)} \dist(y, E) \right\}. \label{def:normalisedhausdorff}
\end{align}
\end{definition}
Moreover, for $x \in \R^n$, $r>0$, we set
\begin{align*}
    \vartheta^d_E(B(x,r)) := \inf \{d_{B(x,r)}(E, L) \,|\, L \mbox{ is a } d \mbox{-dimensional plane in } \R^n\}.
\end{align*}

\begin{definition}
A subset $E \subset \R^n$ is said to be $(d, \epsilon)$-Reifenberg flat (or simply $\epsilon$-Reifenberg flat if the dimension is understood) if 
\begin{align*}
    \vartheta_E^d(B(x,r)) < \epsilon \mbox{ for all } x \in E \mbox{ and } r>0.
\end{align*}
\end{definition}

\begin{definition}[CCBP]\label{def:CCBP}
A coherent collection of balls and planes is a triple
\begin{align*}
(P_0,  \{\Bjk\}, \{\Pjk\}), \, k \in \N, \, j \in J_k,
\end{align*}
where 
\begin{enumerate}[label={\Alph*)}]
\item $P_0$ is a $d$-plane in $\R^n$. 
\item $\{\xjk\}$ is a collection of points in $\R^n$ so that, for fixed $k \in \N$, $\{\xjk\}_{j \in
J_k}$ is a subcollection with the property that
\begin{align}
	\dist(\xjk, \xik) \geq r_k, \,\, r_k = 10^{-k} \enskip \mbox{for all } i,j \in J_k. \label{eq:CCBP1} \tag{CCBP1}
\end{align}
\item $\{\Pjk\}$ is a family of planes so that for each $k \in \N$, $j \in J_k$, $\Pjk \ni \xjk$.
\end{enumerate}
Moreover, $(P_0, \{\xjk\}, \{\Pjk\})$ satisfies the following properties.
\begin{enumerate}[resume, label={\Alph*)}]
\item Set $\Bjk := B(\xjk, r_k)$ and, for $\lambda>0$,
\begin{align*}
\lambda V_k := \bigcup_{j \in J_k} \lambda B_{j, k}.
\end{align*}
The property which we require is that 
\begin{align} 
\xjk \in 2V_{k-1} \mbox{ for all } j \in J_k. \label{eq:CCBP2} \tag{CCBP2}
\end{align}
\item For $j \in J_0$, 
\begin{align}
\dist(x_{j}^0, P_0) \leq \epsilon. \label{eq:CCBP3} \tag{CCBP3}
\end{align}
\item For $k \geq 0$ and for $i,j \in J_k$ such that $\dist(\xjk, \xik) \leq 100r_k$, 
\begin{align}
d_{\xjk, 100r_k} (\Pjk, \Pik) \leq \epsilon. \label{eq:CCBP4} \tag{CCBP4}
\end{align}
\item For $k \geq 0$, $i \in J_k$ and $j \in J_{k+1}$ so that $\dist(\xjk, x_{i,k+1}) \leq 2r_k$, 
\begin{align}
	d_{\xik, 20r_k} (\Pik, P_{j, k+1}) \leq \epsilon.  \label{CCBP5} \tag{CCBP5}
\end{align}
\item For $j \in J_0$, 
\begin{align}
    d_{x_{j}^0, 100}(P_{j,0}, P_0) \leq \epsilon. \label{CCBP6} \tag{CCBP6}
\end{align}
\end{enumerate}
\end{definition}

The main technical result of [DT1] is the following.
\begin{theorem}[{[DT1], Thoerem 2.4}]\label{theorem:davidtoro}
Let $(P_0, \{\xjk\}, \{\Pjk\})$ be a CCBP. Then there exists a bijection 
\begin{align*}
g : \R^n \to \R^n
\end{align*}
so that 
\begin{align*}
\Sigma = g(P_0) 
\end{align*}
is a $C\epsilon$-Reifenberg flat set that contains the accumulation set
\begin{align}
	E^\infty  := \{ x \in \R^n \, & |\, x = \lim_{m \to \infty} x_{j(m)}^{ k(m)}, \nonumber\\
	& k(m) \in \N, \, \, j(m) \in J_{k(m)}, \nonumber \\
	& m \geq 0 \mbox{ and } \lim_{m \to \infty} k(m) =
	\infty \}.
\end{align}
\end{theorem}

David and Toro  also prove a another version of Theorem \ref{theorem:davidtoro}, which requires a further assumption. In order to state it, we introduce one more quantity, which
will be used to control the rate at which the planes belonging to the CCBP change as we go through
locations and scales. We also need to recall from [DT1], Chapter 4, that the function $g$ of Theorem \ref{theorem:davidtoro}, `appears as the limit of a sequence of functions $g_k$'; such sequence converges 
uniformly to $g$. In fact, $g$ is pointwise defined as 
\begin{align} \label{eq:convfunction2}
g(x) := \lim_{k\to \infty} g_k (x), \,\, x \in \R^n. 
\end{align}
Moreover, we have the estimate
\begin{align}
    |g(z) - g_k(z)| \leq C\epsilon r_k \,\, \mbox{ for } k \geq 0 \mbox{ and } z \in P_0. \label{eq:convfunctions}
\end{align}
We refer the reader to Chapters 6, equation (6.1) and Lemma 6.1 in [DT1] for more details on this.  

For $k \geq 1$ and $y \in 11V_k$, set (see [DT1], Chapter 7, equation 7.16)
\begin{align*}
\epsilon_k(y) := \sup \{ d_{\xim, 100r_m} (\Pjk, \Pim) \, | & \, j \in J_k, m \in \{k-1, k\}, i \in
J_m, \\
& \mbox{ and } y \in 11\Bjk \cap \Bim\}.
\end{align*}

The following can be found in [DT], Chapter 8, Proposition 8.3.
\begin{theorem}\label{theorem:davidtoro2}
Let $(P_0, \{\Bjk\}, \{\Pjk\})$ be a CCBP. Let $g$ be the bijection obtained from Theorem
\ref{theorem:davidtoro}. If there exists a number $M>0$ so that
\begin{align}
\sum_{k \geq 0} \epsilon_k(g_k(z))^2 \leq M \,\, \forall z \in P_0, \label{eq:CCBPsum} \tag{CCBPS}
\end{align}
then $g: P_0 \to \Sigma$ is bi-Lipschitz with constant $K$ depending on $n, d$ and $M$.
\end{theorem}
\subsection{Some lemmas on $\betae{p}$}
The following lemmas are to be found in [AS1], Section 2. They will be used later on in the various proofs. 
\begin{lemma}[{[AS1], Lemma 2.11}] \label{lemma:supbeta}
Let $E \subset \R^n$ and let $B(x,t)$ be a ball centered on $E$. Then
\begin{align*}
    \beta_E^{d,1}(x,t) \leq 2^d \beta_{E,\infty}^d (x,t).
\end{align*}
\end{lemma}

\begin{lemma}[{[AS1] Lemma 2.13}] \label{lemma:betaincreasep}
Let $1 \leq p < \infty$, $E \subset \R^n$ a closed subset and $B$ a ball centered on $E$ with $\hdc(B \cap E) > 0$. Then 
\begin{align*}
\betae{1}(B) \lesssim_n \betae{p}(B).
\end{align*}
\end{lemma}

\begin{lemma}[{[AS1] Lemma 2.14}] \label{lemma:monotonicity}
Let $1 \leq p < \infty$. Let $E \subset \R^n$. Let $B(x,t)$ be any ball centered on $E$ and let $B(y, s) \subset B(x,t)$ be also centered on $E$ with $s\leq t$. Then
\begin{align}
    \beta_E^{p,d}(y,s)^p \leq \left(\frac{t}{s}\right)^{d+p} \beta_E^{p,d}(x,t)^p. 
\end{align}
\end{lemma}

\begin{lemma}[{[AS1], Lemma 2.12}] \label{lemma:betap_betainfty}
Let $E \subset \R^n$. Let $B$ be a ball centered on $E$ so that for all $B' \subset B$ also centered on $E$ we have $\hdc(B' \cap E) \geq c_0 r_{B'}^d$. Then 
\begin{align} \label{eq:betap_betainfty}
\beta_{E, \infty}^d (\frac{1}{2}B) \leq 2c_0^{-\frac{1}{d+1}} \betae{1}(B)^{\frac{1}{d+1}}.
\end{align}
\end{lemma}
\begin{remark}
At first sight, one could think that the implication $\ref{item:main_beta} \implies \ref{item:main_tan}$ follows immediately from Lemma \ref{lemma:betap_betainfty}. This is not the case because the infimising planes in \eqref{eq:betap_betainfty} could change as $r\downarrow 0$: we want to prove instead the existence of a \textit{fixed} tangent plane so that \eqref{eq:main_tan} holds true. 
\end{remark}

\begin{lemma}[{[AS1], Lemma 2.21}] \label{lemma:azzamschul}
Let $1 \leq p < \infty$ and $E_1, E_2 \subset \R^n$. Let $x \in E_1$ and fix $r>0$. Take some $y \in E_2$ so that $B(x,t) \subset B(y, 2t)$. Assume that $E_1, E_2$ are both lower content $c$-regular. Then
\begin{align*}
    \beta_{E_1}^{p,d} (x,t) \lesssim_c \beta_{E_2}^{p,d} (y, 2t) + \left( \frac{1}{t^d} \int_{E_1 \cap B(x, 2t)} \left(\frac{\dist (y, E_2)}{t} \right)^p \, d \hdc(y)\right)^{\frac{1}{p}}.
\end{align*}
\end{lemma}

\subsection{Intrinsic cubes and Whitney cubes.}

In this section we briefly recall two `cubes'-tools which will be needed later on. 

\bigskip

\textbf{Intrinsic cubes.} First David in [Da1] and then M. Christ in [Ch1], introduced a construction which allows for a partition of any space of homogeneous type into open subsets which resemble in behaviour the dyadic cubes in $\R^n$. Recall that a space of homogeneous type is a set $X$ equipped with a quasi-metric $p$ and a Borel measure $\mu$ so that 
\begin{itemize}
\item The balls associated to $p$ are open.
\item $\mu(B(x,r)) < \infty$ for all $x \in X$, $r>0$.
\item $\mu$ satisfies the doubling condition with respect to these balls, 
that is, there exists a constant $A$, independent of $x$ and $r$ so that
\begin{align*}
    \mu(B(x, 2r)) \leq A \mu(B(x,r)).
\end{align*}
\end{itemize}
In this case, we use Christ's construction:
\begin{theorem}[Christ's cubes] \label{theorem:christ}
Let $(X, p, \mu)$ be a space of homogeneous type. Then there exists a collection of open subsets $(Q_{j,k})$, $k \in \Z$ and $j \in J_k$, and constants $\delta \in (0,1)$, $C_1 < \infty$ and $a_0>0$ so that
\begin{align}
&  \mu\left(X \setminus \cup_{j \in J_k} \Qjk \right) = 0 \,\, \forall k.\label{christ1}\\
& \mbox{If } k \geq m \mbox{ then either } \Qjk \subset \Qim \mbox{ or } \Qjk \cap \Qim = \emptyset.\label{christ2}\\
& \mbox{For each } (k,j) \mbox{ and each } m < k \mbox{ there is a unique } i \mbox{ so that } \Qjk \subset \Qim. \label{christ3} \\
& \diam(\Qjk) \leq C_1 \delta^k \label{christ4}\\
& \mbox{Each } \Qjk \mbox{ contains some ball } B(\xjk, a_0 \delta^k). \label{christ5}
\end{align}
\end{theorem}
We refer the reader to [Ch1] for a precise definition of $\xjk$. Roughly speaking, this is the center of the `cube' $\Qjk$. Also, the original statement involves a further property which we will not use.
T. Hyt\"{o}nen and H. Martikainen provided a variant to this construction in [HM1]. The main difference between the two, is that the latter results in an exact partition of the space $(X, p, \mu)$, that is 
\begin{align*}
    \bigcup_{j \in J_k} \Qjk = X\,\, \forall k \in \Z.
\end{align*}
However, Hyt\"{o}nen and H. Martikainen's cubes are not open. Taking their interior, however, one recovers Christ's cubes. 

We will write
\begin{align*}
    \mathcal{D} := \bigl\{ \Qjk \, |\, k \in \mathbb{Z} \mbox{ and } j \in J_k\bigr\}.
\end{align*}
To consider the family of cubes of a given scale, we fix $k \in \mathbb{Z}$ and we write
\begin{align*}
    \mathcal{D}_k :=\bigl\{ \Qjk \, |\, j \in J_k\bigr\}.
\end{align*}

We remark that in what follows we will not have the measure $\mu$. We will only have a distance (the Euclidean one). But this is not a problem, the construction of the cubes $Q_{j,k}$ does not depend on $\mu$. 

\bigskip

\textbf{Whitney cubes.}
We follow [St1], Lemma 2 of Chapter 1. Let $F$ be a closed nonempty subset of $\R^n$ endowed with the standard euclidean distance. Then there exists a collection $\{S_k\}$ of closed cubes with disjoint interiors which covers $F^c$ (the complement of $F$) and whose side lengths are comparable to their distance to $F$, that is, there exists a constant $A$ so that
\begin{align*}
    A^{-1} \dist(S_k, F) \leq \ell(S_k) \leq A \dist(S_k, F).
\end{align*}

\section{Equivalence of \eqref{item:main_tan}, \eqref{item:main_hausdorff} and \eqref{item:main_approx}}
In this section we prove the equivalence ($\hd$-almost everywhere) of the different versions of tangent present in Thereom \ref{theorem:main}. As mentioned above, these equivalences are already known; we add proofs for the sake of completeness.
\begin{lemma}
A point $x \in E$ is an approximate tangent point of $E$  if and only it's a tangent point of $E$. 
\end{lemma}
\begin{proof}
Take $E$ to be $d$-lower content regular (this insures that the hypothesis on the nonvanishing upper density is satisfied). Let us first show that if a point $x \in E$ is a tangent point in the sense of \eqref{eq:tangent}, then there exists an approximate tangent at $x$. Indeed, suppose not. Then there exists an $\epsilon >0$, a sequence $r_j \downarrow 0$  and an $0<s<1$ so that 
\begin{align*}
    \hd(\{E\cap B(x, r_j) \setminus X(x, V, s)\}) > \epsilon r_j \enskip \enskip \forall j \in \N.
\end{align*}
This implies that for each $j \in \N$, we may find a point $y_j \in E \cap B(x, r_j)$ but not in $X(x, V, s)$ such that
\begin{align*}
    \dist(y_j, V) \gtrsim s r_j.
\end{align*}
This is a contradiction, since we assumed that $a$ is a tangent point (in the sense of \eqref{eq:tangent}. 

Let us prove the converse. Let $x \in E$ be a point which admits an approximate tangent $V$. Again, we argue by contradiction; assume that there exists an $\epsilon >0$ and a sequence of radii $r_j \downarrow 0$, such that 
\begin{align*}
    \sup_{y \in B(x, r_j) \cap E} \frac{ \dist(y, V)}{r_j} > \epsilon.
\end{align*}
Then there exists a sequence of points $y_j \in E \cap B(x, r_j)$ such that $\dist(y_j, V)>\frac{\epsilon}{2} r_j$. Moreover, there exists an $s \in (0,1)$, depending on $\epsilon$, such that 
\begin{align*}
    B(y_j, \eta r_j) \subset B(x, r_j) \setminus X(a, V,  s),
\end{align*}
with $\eta= \epsilon/2$, and, since $E$ is $d$-lower content regular, 
\begin{align*}
    \hdc(B(y_j, \eta r_j) \cap E) > c_0 (\eta r_j)^d.
\end{align*}
Thus, 
\begin{align*}
    \hd(E \cap B(x, r_j) \setminus X(x, V, s)) & \geq \hdc(E \cap B(x, r_j) \setminus X(a,V, s)) \\
    & \geq \hdc(B(y_j, \eta r_j) \cap E)\\
    & \geq c_0 (\eta r_j)^d. 
\end{align*}
This contradicts the fact that $V$ is an approximate tangent at $x$. 
\end{proof}

\begin{lemma}
For $\hd$-almost every point $x \in E$, being a tangent point of $E$ is equivalent of being an $H$-tangent point of $E$.
\end{lemma}
\begin{proof}
Note that if $x$ is an $H$-tangent point, that is, it satisfies \eqref{eq:main_hausdorff}, than it follows immediately that it's a tangent point, i.e. it satisfies \eqref{eq:main_tan}.
Let us prove the converse statement; set $\mathcal{T}:= \{x \in E \, |\, x \mbox{ is a tangent point of } E\}$; without loss of generality, we assume that $0 < \hd(\mathcal{T})<\infty$. On the other hand, set $\mathcal{HT}:= \{x \in E \, |\, x \mbox{ is an } H-\mbox{tangent point of } E\}$. We already know that $\mathcal{HT} \subset \mathcal{T}$ and we want to show that $\hd\ps{\mathcal{T} \setminus \mathcal{HT}} = 0 $.
Let $f: \R^d \to \R^n$ be one of the Lipschitz maps  covering $E_0$ and set 
\begin{align}
    M:= f^{-1}(E_0).
\end{align}
Then $M$ is a measurable subset (with respect to the $d$-dimensional Lebesgue measure on $\R^d$). It suffices to prove that $\hd$-almost every point $x \in f(M)$ is an $H$-tangent point of $E$. Let $p \in M$ be a density point and such that $f$ is differentiable here. Let us denote the tangent plane of $f$ at $p$ by $L_p$. As $p$ is a density point, then we have that for all $\epsilon>0$, there exists a radius $r_0>0$ such that for any $q \in B^n(p, r)$, with $0<r<r_0$, we can find a $q' \in M$ so that $|q-q'|< \epsilon r$. 
From this it follows that, given $\epsilon>0$ and for $r>0$ sufficiently small, for any point $y \in f(B(p, r))$, we may find a point $y' \in f(M)\subset E_0$ so that $|y-y'| \lesssim \epsilon r$. To prove the lemma we can now proceed as follows. Given $\epsilon>0$, let $r>0$ be small enough. Given $q \in L_p$, let $y_q \in f(B(p, r))$ be so that 
$|y_q-q| \leq 2 \dist(q, f(B(p, r)))$. Then we have 
\begin{align*}
    \inf_{x \in E} |q - x| \lesssim \dist(q, f(B(p,r)) + \inf_{x \in E} |y_q-x|.
\end{align*}
There exists a point $y \in f(M)$ so that $|y-y_q|\lesssim \epsilon r$, and so,  $\inf_{x \in E} |y_q - x| \leq \inf_{x \in E} |x-y| + |y-y_q|$; note that the first term on the right hand side of this last inequality vanishes, since $y \in E$. If we now supremize over $q \in L_p \cap B(x_0, r)$, we obtain
\begin{align}
    \sup_{q \in L_p\cap B(x_0, r)} \dist(q, E) \lesssim \sup_{ q \in L_p\cap B(x_0, r)} \dist(q, f(B(p,r))) + \epsilon r.
\end{align}
Because $f$ is differentiable at $p$, $\sup_{ q \in L_p\cap B(x_0, r)} \frac{\dist(q, f(B(p,r)))}{r} \to 0$ as $r \to 0$. This completes the proof of the lemma.
\end{proof}
\section{\eqref{item:main_beta} implies \eqref{item:main_tan}} 

\begin{proposition} \label{theorem:summabletangent}
Let $E \subset \R^{n}$ be $d$-LCR with constant $c_0$; assume moreover that $E \subset B(0,1)$. For $k \in \Z$ set $r_k = 10^{-k}$. Let $1 \leq p < \infty$. Then $\hd$-almost every point $x \in E$ such that
\begin{align*}
\sum_{k \in \N} \beta_E^{p, d} (x, r_k)^2 < \infty,
\end{align*}
is a tangent point of $E$ (in the sense of \eqref{eq:tangent}).
\end{proposition}

Let us give some easy remarks.
\begin{remark}
Because of Lemma \ref{lemma:betaincreasep}, it is enough to prove Theorem \ref{theorem:summabletangent} for $p=1$. 
We may assume that the set
\begin{align*}
   \bigl\{ x \in E\, |\, \sum_{k \in \N} \betae{1} (x, r_k)^2 < \infty \bigr\}
\end{align*}
has non zero $\hd$ measure, for otherwise there is nothing to prove. 
\end{remark}

\begin{lemma} \label{lemma:reduction1}
It suffices to show that if $A \subset \bigl\{ x \in E\, |\, \sum_{k \in \N} \betae{1} (x, r_k)^2 < \infty \bigr\}$ and $0 < \hd(A) < \infty$, then $\hd$-almost every $x \in A$ is a tangent point of $E$.
\end{lemma}
\begin{proof}
Suppose that if $A \subset \bigl\{ x \in E\, |\, \sum_{k \in \N} \betae{1} (x, r_k)^2 < \infty \bigr\}$ and $0 < \hd(A) < \infty$, then $\hd$-almost every $x \in A$ is a tangent point of $E$. We then claim that for any $A' \subset E$ with $\sum_{k \in \N} \betae{1}(x, r_k)^2 < \infty$ on $A'$, then $\hd$-almost every $x \in A'$ is a tangent point of $E$ (with no assumption on the $d$-dimensional Hausdorff measure of $A'$).

We argue by contradiction. Suppose that there exists a subset $\Omega \subset A'$ with $0 < \hd(\Omega)$ and so that no $x \in \Omega$ is a tangent point of $E$. If $\hd(\Omega) <\infty$ then we already reached a contradiction. Hence we may take $\hd(\Omega) = \infty$. However, by Theorem 8.13 in [M1], there exists a compact subset $K \subset \Omega$ with $0< \hd(K) < \infty$ and so that no $x \in K$ is a tangent point of $E$. This leads again to a contradiction. Thus the lemma follows.
\end{proof}

\begin{remark} \label{remark:e0}
Let now $A \subset \bigl\{ x \in E\, |\, \sum_{k \in \N} \betae{1} (x, r_k)^2 < \infty \bigr\}$ be an arbitrary compact subset such that $0< \hd(A) < \infty$. Again, by Theorem 8.13 in [M1], such set exists. 
Let $\epsilon >0 $ be given. For $\ell \in \N$, define
\begin{align*}
    E_\ell:= \Bigg\{ x \in A\, |\, \sum_{k \geq \ell}^\infty \betae{1}(x, r_k)^2 < \epsilon^2\Bigg\}.
\end{align*}
For an arbitrary $\ell \in \N$ with $\hd(E_\ell) >0$, we now fix 
\begin{align} \label{eq:e_0}
    E_0 \subset E_\ell \subset A 
\end{align}
to be a compact subset of positive and finite $\hd$-measure. Notice that if we prove Proposition \ref{theorem:summabletangent} for $E_0$, then the statement in its full generality follows by first saturating $E_\ell$ with compact subsets, and then by taking the countable union $\cup_\ell E_\ell = A$. Moreover, by applying a dilation and a translation, we may assume that 
\begin{align} \label{eq:assumptioscaleplace}
E_0 \subset B(0, 1), \mbox{ with } 0 \in E_0,
\end{align}
and that $\ell = 1$ in \eqref{eq:e_0}. Thus, we see that, by construction
\begin{align}
    \betae{1}(x, r_k) < \epsilon \mbox{ for } x \in E_0 \mbox{ and for } k \geq 1.
\end{align}
\end{remark}

\begin{remark}
Note that $E_0 \subset E$ will not necessarily be a $d$-lower content regular set. However, this will not cause issues, for when estimating angles, we will be always integrating over the whole $E$.
\end{remark}

We now show that one can construct a bi-Lipschitz parameterisation of any such a subset $E_0$. \begin{remark}
\textbf{\textit{Its proof follows very closely [DT1], Chapter 13: the only difference being that we are dealing with a different kind of $\beta$ numbers.}}
\end{remark}

\begin{lemma} \label{lemma:param}
Fix $\epsilon>0$ and let $E_0 \subset B(0,1)$ be a compact subset of $\R^n$ with $0 \in E_0$ and such that $0 <\hd(E_0) < \infty$. If for $\hd$-almost every point $x$ in $E_0$, we have
\begin{align}
    \sum_{k=0}^\infty \betae{1} (x, r_k)^2 < \epsilon^2, 
\end{align}
then, denoting by $P_0$ the plane which minimises $\betae{1}(0,1)$, there exists a bi-Lipschitz map $g: P_0 \to \R^n$ with constant $\Lip(g) \leq Ce^{\epsilon}$ and such that $E_0\subset g(P_0)$.
\end{lemma}

We will prove Lemma \ref{lemma:param} in the next few subsections; the proof will consists in checking the hypotheses of Theorem \ref{theorem:davidtoro}.

\subsubsection{Construction of a CCBP: A to E}

Consider a maximal collection $\{\ujk\}$ of points in $E_0$ such that
\begin{align*}
\dist(\ujk, \uik) \geq \frac{4}{3}r_k \enskip \enskip i \neq j;\,\, i,j \in J_k.
\end{align*}
We may assume that $\#(J_0) = 1$ and that $u_{j, 0} = 0$. This is because of \eqref{eq:assumptioscaleplace}.

From now on, we will denote $L(\ujk, r_k)$, that is, the plane which infimises $\betae{1}(\ujk, 120r_k)$, as $\Ljk$.

Fix $k,j$. We choose a point $\xjk \in E$ closest to $\Ljk$ in $E \cap B(\ujk, r_k/3)$. Note that 
\begin{align}
     E \cap B(\ujk, r_k/3) 
    & \subset \{y \in E\cap B(\ujk, 120r_k) \, |\, \dist(y, \Ljk) \geq \dist(\xjk, \Ljk)\} \label{eq:60}
\end{align}
This and Chebyshev inequality give us
\begin{align}
    & \dist(\xjk, \Ljk) \\
    & \leq \frac{1}{\hdc(B(\ujk, r_k/3)\cap E)} \int_{E \cap B(\ujk, 120r_k)} \dist(y, \Ljk) \, d\hdc(y) \label{eq:62}\\
    & \lesssim \frac{1}{r_k^d} \int_{E \cap B(\ujk, 120r_k)} \dist(y, \Ljk) \, d\hdc(y) \label{eq:63} \\
    & = \betae{1} (\ujk, 120r_k) r_k \label{eq:64}\lesssim \epsilon r_k.
\end{align}
 Inequality \eqref{eq:62} follows from the containment \eqref{eq:60}. Inequality \eqref{eq:63} follows from the lower content regularity of $E$. Equation \eqref{eq:64} is due to the choice $\ujk \in E_0$.

We now define $\Pjk$ to be the $d$-plane parallel to $\Ljk$ such that it contains $\xjk$. The family $\{\Pjk\}$ will constitute the collection of planes through which we will be able to verify David and Toro's Theorem. 

\begin{enumerate}[label={\Alph*)}] 
\item There is only one element in the family $\{x_{j}^0\}_{j \in J_0}$. Thus we take 
\begin{align} \label{eq:plane0def}
P_0 := P(x_{j}^0, 1)
\end{align}
in the definition of CCBP.

\item The family $\{\xjk\}$ satisfies \eqref{eq:CCBP1} by construction. This is true up to a constant, but this is unimportant.

\item $\{\Pjk\}_{k\in \N, j\in J_k}$ will be the sought for family of $d$-planes.

\item Notice that 
\begin{align*}
E_0 \subset \bigcup_{j \in J_k} \overline{B}\left(\ujk, \frac{4}{3}r_k\right).
\end{align*}

Moreover, since $\xjk$ were chosen to be at distance strictly less than $\frac{1}{3}r_k$ from $\ujk$, then 
\begin{align*}
    \overline{B}\left(\ujk, \frac{4}{3}r_k\right) \subset B\left(\xjk, \frac{5}{3} r_k \right).
\end{align*}
Thus
\begin{align*}
    E_0 \subset \bigcup_{j \in J_k} B\left(\xjk, \frac{5}{3} r_k\right).
\end{align*}   
This lets us conclude that for $u \in E_0$, we may pick $i \in J_{k-1}$ so that $u \in B(x_{i}^{k-1}, \frac{5}{3}r_{k-1})$, and so 
\begin{align*}
    \dist(\xjk, x_{i}^{ k-1})  \leq \dist(\xjk, E_0) + \dist(u, x_{i}^{k-1})
    \leq \frac{1}{3}r_k + \frac{5}{3}r_{k-1} 
    < 2r_{k-1},
\end{align*}
and therefore that
$
    \xjk \in 2V_{k-1}.
$

\item Clearly \eqref{eq:CCBP3} follows from \eqref{eq:plane0def}.

\end{enumerate}

\begin{remark} \label{remark:AE}
We have now verified that $E_0$ (as given in Remark \ref{remark:e0}) satisfies hypothesis A, B, C, D, and E in Definition \ref{def:CCBP}.
\end{remark}

\begin{remark}
It is important to notice that $E_0 \subset E^\infty$ (where here $E^\infty$ is as in Theorem \ref{theorem:davidtoro}). Indeed, let $x \in E_0$. Then there is a sequence $m \mapsto u_{j(m)}^{k(m)}$ so that $\lim_{m\to \infty} u_{j(m)}^{k(m)} = x$. But then one also have that
\begin{align*}
    x_{j(m)}^{ k(m)} \to x \,\, \mbox{ as } \,\, m \to \infty.
\end{align*}
Thus $E_0 \subset E^\infty$. 
\end{remark}

\subsubsection{Construction of a CCBP: F to H}

We start with the following, somewhat technical, lemma, to be found in [DT1], Chapter 12, Lemma 12.7. It is useful to prove Lemma \ref{lemma:anglecontrol} below, which lets us control the angle between planes belonging to different location and scales.

\begin{remark}
    We remark that [AS1] provides a way of controlling angles between planes through $\beta$ numbers, that is, Lemma 2.16. However the construction which follows, which is the one given in [DT1], Chapter 13, adapted to our current $\beta$'s, is needed to check the summability hypothesis \eqref{eq:CCBPsum}.
\end{remark}

\begin{lemma} \label{lemma:planesdistance}
Let $x \in \R^n$, $r>0$, $\tau \in (0, 10^{-1})$. Let $P_1, P_2$ be two affine $d$-planes and take
$d$ orthogonal unit vectors $\{e_1,...,e_d\}$. Suppose that, for $0 \leq l \leq d$, we are given
points $a_l \in P_1$ and $b_l \in P_2$, so that 
\begin{align*}
& a_0 \in B(x, r), \\
& \dist(a_l,b_l) \leq \tau r \, \, \mbox{ for } 0 \leq l \leq d, \\
& \dist(a_l - a_0, re_l) \leq \frac{r}{10} \, \, \mbox{ for } 1 \leq l \leq d.
\end{align*}
Then 
\begin{align*}
d_{x, s} (P_1, P_2) \leq C \tau \, \, \mbox{ for } r \leq s \leq 10^4r. 
\end{align*}
\end{lemma}
\begin{proof}
See [DT1], Chapter 12, Lemma 12.7.
\end{proof}

\begin{lemma} \label{lemma:anglecontrol}
Let $k \in \N$ and $j \in J_k$. Take either $m=k$ or $m=k-1$ and let $i \in J_m$ be so that
\begin{align} \label{eq:CCBPfh3}
    \dist(\xjk, \xim) \leq 100r_m.
\end{align}
Then, for $\frac{r_k}{2} \leq s \leq 5000r_k$, 
\begin{align}
d_{\xjk, s} (\Pjk, \Pim) &  \lesssim \betae{1}(\ujk, 120r_k) + \betae{1}(\uim, 120r_m),  \label{eq:hausdorffbeta}
\end{align}
\end{lemma}

\begin{proof}
Fix $k, j \in J_k$. We will now choose $d+1$ points which will control the position of $\Ljk$.
Let $\{e_1,...,e_d\}$ be an orthonormal basis for the vector space $V \in \grass$ parallel to $\Ljk$. Let
\begin{align*}
     p_0 = \ujk, \, \, 
    p_1 := p_0 + \frac{r_k}{2} e_1, 
    \, \,  \hdots, \,\,
    p_d := p_0 + \frac{r_k}{2} e_d.
\end{align*}
For $0 \leq l \leq d$ it holds that $\hdc(E \cap B(p_l, C\epsilon r_k)) \geq C \epsilon^d r_k^d$. Thus we may find a point 
\begin{align}
    w_l \in E \cap B(p_l, C\epsilon r_k). \label{eq:65b}
\end{align}
Below it will be useful to know that
\begin{align}
    B(w_l, r_{k+2}) \subset B(\xjk, r_k). \label{eq:66}
\end{align}
To see this, notice that
\begin{align*}
    |w_l - \xjk| & < |w_l - \ujk| + |\ujk- \xjk| \\
    & \leq |w_l - p_l| + |p_l - \ujk| + |\ujk - \xjk| \\
    & \leq C\epsilon r_k + \frac{1}{2}r_k + \frac{1}{3}r_k \\
    & = \frac{5}{6}r_k + C\epsilon r_k.
\end{align*}
Thus \eqref{eq:66} holds as long as $C\epsilon < 1/6$. By the choice of $\epsilon$, this is certainly satisfied. 

Again by using the lower content regularity of $E$ and Chebyshev inequality, we may find, for $0 \leq l \leq n$, a $z_l \in E\cap B(w_l, r_{k+2})$ which is the closest there to $\Ljk$ and thus, as in \eqref{eq:62} - \eqref{eq:64}, satisfies
\begin{align}
    \dist(z_l, \Ljk) & \leq \frac{1}{\hdc(E \cap B(w_l, r_{k+2}))} \int_{E\cap B(w_l,r_{k+2})} \dist(y, \Ljk) \, d \hdc(y)\nonumber \\
    & \lesssim \betae{1}(\ujk, 120r_k) r_k \lesssim \epsilon r_k. \label{eq:67}
\end{align}
Notice that we used the containment $B(w_l, r_{k+2}) \subset B(\xjk, r_k) \subset B(\ujk, 120r_k)$. 

Keeping $k,j$ fixed, take 
\begin{align*}
    \mbox{either   } m = k  \mbox{   or   } m = k-1,
\end{align*}
and let $i \in J_m$ be such that 
\begin{align*} 
\dist(\xjk, \xim) \leq 100r_m.
\end{align*}

Notice that 
\begin{align}
    B(w_l, r_{k+2}) \subset B(\uim, 110r_m). \label{eq:70}
\end{align}
Indeed, 
\begin{align*}
     |\uim- w_l| 
    & \leq |\uim - \xim| + |\xim - \xjk| +|\xjk - \ujk| + |\ujk- w_l|\\
    & \leq \frac{r_m}{3} + 100r_m + \frac{1}{3}r_k  + \left( \frac{1}{2}r_k+ C\epsilon r_k\right) \leq 110r_m.
\end{align*}

Because of \eqref{eq:70}, using lower content regularity (i.e. \eqref{eq:lowercontent}) we may  choose $z_l \in B(w_l, r_{k+2})$ so that we also have
\begin{align}
& \dist(z_l, \Lim) \\
& \lesssim \frac{1}{\hdc(E \cap B(w_l, r_{k+2}))} \int_{E \cap B(w_l, r_{k+2})}
\dist(y, \Lim) \, d \hdc(y) \nonumber \\
& \lesssim \frac{1}{r_k^d} \int_{E \cap B(\uim, 120r_m)} \dist(y, \Lim) \, d \hdc(y) \nonumber\\
& = \betae{1}(\uim, 120r_m). \label{eq:68}
\end{align}

We proceed as follows. Take $x=\xjk$ in Lemma
\ref{lemma:planesdistance}, take $r=\frac{r_k}{2}$ and take 
\begin{align*}
\tau = \betae{1}(\ujk, 120r_k) + \betae{1}(\uim, 120r_m).
\end{align*}
We will let $P_1$ to be $\Pik$ and $P_2$ to be $\Pim$. Moreover we take $\{e_1,...,e_d\}$ to be the
orthonormal basis of the $d$-subspace parallel to $\Pjk$ which was used to select the $z_l$'s.
Recall that
\begin{align*}
& \dist(z_l, \Ljk) \lesssim \betae{1}(\ujk, 120r_k)r_k \lesssim \epsilon r_k \\
& \dist(z_l, \Lim) \lesssim \betae{1}(\uim, 120r_m)r_m \lesssim \epsilon r_m. 
\end{align*}
Let $q_l \in \Ljk$ be so that $\dist(z_l, \Ljk) = \dist(z_l, q_l)$. 
We have
\begin{align*}
\dist(z_l, \Pjk) & = \inf_{p \in \Pjk} |z_l - p| \\
& \leq  |z_l - q_l| + \inf_{p \in \Pjk} |q_l - p| \\
& = \dist(z_l, \Ljk) + \dist(q_l, \Pjk) \\
& \lesssim \betae{1}(\ujk, 120r_k)r_k \\
& \lesssim \epsilon r_k.
\end{align*}
The estimate $\dist(q_l, \Pjk) \lesssim \betae{1}(\ujk, 120r_k)r_k$ follows from our initial choice of $\{\xjk\}$ and recalling that $\Pjk$ was constructed by translating $\Ljk$ by a distance smaller than $\betae{1}(\ujk, 120r_k)$.
We similarly obtain
\begin{align*}
    \dist(z_l, \Pim) \lesssim \betae{1}(\uim, 120r_m)r_m \lesssim \epsilon r_m. 
\end{align*}

For $0 \leq l \leq d$, set $a_l$ to be the point in $\Pjk$ for which 
$$|z_l - a_l| = \inf_{a \in \Pjk} |z_l - a|,$$
and similarly, let $b_l \in \Pim$ so that $|z_l- b_l| = \dist(z_l, \Pim)$.
Then,
\begin{align}\label{eq:CCBPfh1}
|a_l-b_l| & \leq |a_l - z_l| + |z_l - b_l| \nonumber\\
& = \dist(z_l, \Pjk) + \dist(z_l, \Pim) \nonumber \\
& \lesssim \betae{1}(\ujk, 120r_k)r_k + \betae{1}(\uim, 120r_m)r_m \\
& \leq \tau r_m. 
\end{align}
In \eqref{eq:CCBPfh1} we used \eqref{eq:67} and \eqref{eq:68}.

For $0 \leq l \leq d$, we have
\begin{align} \label{eq:CCBPfh2}
|(a_l - a_0) - \frac{r_k}{2}e_l| \leq \frac{r_k}{20}. 
\end{align}
Indeed, notice that
\begin{align*}
    |(a_l - a_0) -  \frac{r_k}{2} e_l| & \leq |(a_l - a_0) - (z_l - z_0)| + |z_l -z_0 - \frac{r_k}{2} e_l|\\
    & = |a_l - a_0 - (z_l - z_0)| + |z_l -z_0 - (p_l - p_0)|. 
\end{align*}
Now, 
\begin{align*}
|z_l -z_0 -( p_l - p_0)|  \leq 2r_{k+2} + |w_l - w_0 - (p_l - p_0)| 
     < 3r_{k+2}.
\end{align*}
This is due to how we defined $p_l$'s and to \eqref{eq:65b}.
One the other hand, 
\begin{align*}
|a_l - a_0 - (z_l - z_0)| \leq |a_l - z_l| + |a_0 - z_0|  \lesssim \epsilon r_m.
\end{align*}
Since $\epsilon$ is small, \eqref{eq:CCBPfh2} is satisfied. 

Hence we may apply Lemma \ref{lemma:planesdistance}, to obtain that
\begin{align*}
d_{\xjk, s} (\Pjk, \Pim)   \lesssim \tau 
= \betae{1}(\ujk, 120r_k) + \betae{1}(\uim, 120r_m),
\end{align*}
for $\frac{r_k}{2} \leq s \leq 5000r_k$. 
\end{proof}

\begin{enumerate}[resume, label={\Alph*)}]
\item We now obtain \eqref{eq:CCBP4}: that is, for $k \geq 0$ and for $i,j \in J_k$ such that 
\begin{align} \label{eq:CCBP4b}
\dist(\xjk, \xim) \leq 100r_k,
\end{align}
one wants $\dist_{\xjk, 100r_k}(\Pjk, \Pik) \lesssim \epsilon$. Indeed, \eqref{eq:CCBP4b} coincides with the premise \eqref{eq:CCBPfh3} of Lemma \ref{lemma:anglecontrol}, which can therefore be applied with $k=m$ and choosing $s =100r_k$, to obtain 
\begin{align*}
d_{\xjk, 100r_k}(\Pjk, \Pik) \lesssim \betae{1}(\ujk, 120r_k)
\end{align*}
and thus $\dist_{\xjk, 100r_k}(\Pjk, \Pik) \lesssim \epsilon$, since $\ujk \in E_0$.

\item We  verify \eqref{CCBP5} with $k-1$: first, notice that \eqref{eq:CCBPfh3} is satisfied since we only need $\dist(\xjk, x_{i,k-1})< 2r_{k-1}=2r_m$. Then we apply \eqref{eq:hausdorffbeta} with the choice $s=20r_{k-1}$ so to have
\begin{align*}
d_{x_{j,k-1}, r_{k-1}} (P_{j,k-1}, P_{i,k}) \lesssim \epsilon.
\end{align*}

\item To see \eqref{CCBP6}, we apply Lemma \ref{lemma:anglecontrol} with $k=m=0$ and we recall that $\diam(E_0) <1$, thus \eqref{eq:CCBPfh3} is immediately satisfied. Moreover, recall from \eqref{eq:plane0def} and the line below, that $P_0$ is in fact $P_{j,0}$ for some $j \in J_0$. Then \eqref{CCBP6} follows.
\end{enumerate}

\begin{remark}
Thus the properties F, G and H listed in Definition \ref{def:CCBP} are satisfied. This, together with Remark \ref{remark:AE} let us apply Theorem \ref{theorem:davidtoro}: we obtain the covering function $g$. Furthermore, $\Sigma \supset E_\infty$, and thus 
\begin{align*}
    E_0 \subset g(P_0).
\end{align*}
\end{remark}

\subsubsection{Summability condition \eqref{eq:CCBPsum}.} 
Recall $f_k, f$ from \eqref{eq:convfunction2} and the paragraph above. 
For $z \in P_0$, $y = f_k (z)$ and $y \in 10B_{j,k} \cap 11B_{i,m}$, we need to control 
\begin{align*}
\epsilon_k(y) := \sup_{m \in \{k-1,k\}, i \in J_m} d_{x_{i,m}, 100r_m}(\Pjk, \Pim). 
\end{align*}
Pick $\tilde z \in E_0$ so that 
\begin{align*}
    |f(z) - \tilde z| = \dist(f(z), E_0).
\end{align*}
We notice that
\begin{align}
    |\tilde z - y| & \leq |\tilde z - f(z)| + |f(z) - f_k(z)| \nonumber \\
    & = \dist(f(z), E_0) + |f(z) - f_k(z)| \label{eq:610}\\
    & \leq |f(z) - f_k(z)| + \dist(f_k(z), E_0) + |f(z) - f_k(z)| \nonumber\\
    & \leq C\epsilon r_k + \dist(f_k(z), E_0) \label{eq:611}\\
    & \leq C \epsilon r_k + |f_k(z) - \ujk| \label{eq:612}\\
    & \leq C\epsilon r_k + |f_k(z) - \xjk| + |\xjk - \ujk| \\
    & \leq C \epsilon r_k + 10r_k + \frac{r_k}{3}\label{eq:613} \\
    & \leq 11r_k.
\end{align}
Equation \eqref{eq:610} is due to the choice of $\tilde z$. \eqref{eq:611} is by \eqref{eq:convfunctions}. Inequality \eqref{eq:612} is also due to the choice of $\tilde z$ and by the fact that $\ujk \in E_0$. For \eqref{eq:613} recall that $y= f_k(z) \in 10\Bjk$, and that $\dist(\ujk, \xjk)\leq \frac{1}{3}r_k$.

Hence, 
\begin{align*}
|\ujk - \tilde z| & \leq |\ujk - \xjk| + |\xjk - y| + |y - \tilde z|  \leq 10r_k + \frac{1}{3}r_k + 11r_k \leq 22r_k.
\end{align*}
Similarly, since $y \in 11\Bim$, 
\begin{align*}
|\uim - \tilde z| & \leq |\uim - y| + |y - \tilde z| \leq 23r_m.
\end{align*}
This implies that
\begin{align*}
    B(\ujk, 120r_k) \cup B(\uim, 120r_m) \subset B(\tilde z, r_{k-3}). 
\end{align*}
Let $L(\tilde z, r_{k-3})$ be the $d$-plane which infimizes $\betae{1}(\tilde z, r_{k-3})$. Then 
\begin{align*}
    & \betae{1}(\ujk, 120r_k) + \betae{1}(u_i^m, 120r_k) \\ 
    & = \frac{1}{120r_k^d} \int_{E \cap B(\ujk, 120r_k)} \frac{\dist(y, L(\ujk, 120r_k))}{120r_k} \, \hdc(y)\\
    & \enskip \enskip + \frac{1}{120r_m^d}\int_{E \cap B(\uim, 120r_k)} \frac{\dist(y, L(\uim, 120r_k))}{120r_m}\, \hdc(y)\\
    & \leq \frac{1}{120r_k^d} \int_{E \cap B(\tilde z, r_{k-3})} \frac{\dist(y, L(\tilde z, r_{k-3}))}{120r_k} \, \hdc(y) \\
    & \enskip \enskip + \frac{1}{120r_m^d} \int_{E \cap B(\tilde z, r_{k-3})} \frac{\dist(y, L(\tilde z, r_{k-3})}{120r_m}\, \hdc(y)\\
    & \lesssim \betae{1}(\tilde z, r_{k-3}).
\end{align*}

All in all, we have the sequence of inequalities
\begin{align*}
d_{\xim, 100r_k}(\Pjk, \Pim) & \leq d_{\xjk, 200r_k}(\Pjk, \Pim)\\
& \lesssim \betae{1}(\xjk, r_k) + \betae{1}(\xim, 120r_m) \\
& \lesssim \betae{1}(\tilde z, r_{k-3}),
\end{align*}
that is, 
\begin{align*}
    d_{\xim, 100r_k}(\Pjk, \Pim) & \leq d_{\xjk, 200r_k}(\Pim, \Pjk) \leq C\betae{1}(\tilde z, r_{k-3}).
\end{align*}
Notice that the constant $C$ here does not depend on $k$ or $m$, and similarly for the choice of $\tilde z$. Hence we may write
\begin{align*}
    \epsilon_k(f_k(z)) \leq \betae{1}(\tilde z, r_{k-3}) \,\, \mbox{ for } z \in P_0, \, \tilde z \in E_0.
\end{align*}
This implies immediately that the summability condition \eqref{eq:CCBPsum} is satisfied. 

Having also verified the summability condition \eqref{eq:CCBPsum}, we can apply Theorem \ref{theorem:davidtoro2} to $E_0$ and thus obtain the bi-Lipschitz parameterisation $g$. This concludes the proof of Lemma \ref{lemma:param}. 

\begin{remark}
Clearly, at $\hd$-almost every point $x $ of $E_0$, we can find a tangent $d$-plane (in the sense of \eqref{item:main_tan}). However, a priori we can't say that such a tangent is a tangent  E. The next lemma uses the decay of the $\beta$ coefficients at $\hd$-a.e. $x \in E_0$ to show that, indeed, this turns out to be the case.
\end{remark}

\begin{lemma}
Let $E$ be a $d$-LCR subset of $\R^n$; consider a subset $E_0 \subset E$ with $0< \hd(E_0) < \infty$ which can be covered by a bi-Lipschitz image of a $d$-dimensional plane $P_0$. If at $\hd$-almost every point $x$ in $E_0$ it holds that 
$\sum_{k=0}^\infty \betae{1}(x,r)^2 < \epsilon^2$, then $x$ is a tangent point of $E$ \textup{(}in the sense of \eqref{item:main_tan}\textup{)}.
\end{lemma}
\begin{proof}
Without loss of generality, let us assume that $P_0 = \R^d$.
Set $M:= g^{-1}(E_0)$; note that $M$ is a measurable set with respect to the $d$-dimensional Lebesgue measure on $\R^d$ and clearly $0<\mathcal{L}^d(M)< \infty$. Let $p_0 \in M$ be a density point and moreover assume that $g$ is differentiable there | that the set of points in $M$ with these characteristics has full measure. It follows from the Lebesgue differentiation theorem that for all $\epsilon>0$, there exists an $r_0>0$ such that if $p' \in B^d(p_0, r)$, $0<r<r_0$, then there exists a point $p \in M$ such that $|p-p'|< \epsilon r$. 

Set $x_0 = g(p_0)$. Because $g$ is bi-Lipschitz, there exists a constant $c\leq 1$, depending only on $\Lip(g)$, such that $B^d(p_0, cr)  \subset g^{-1}(B^n(x_0, r))$; furthermore, since we chose $p_0$ to be a density point, for any $y \in g(B^d(p_0,cr))$, there exists a $y' \in f(M)$ so that $|y-y'|\leq C \epsilon r$. In particular, we can find $d$ points $y_1,...,y_d \in f(M)$ such that they are linearly independent (with good constant) and such that $\dist(y_j, L_p) \leq C \epsilon r$ for all $j=1,...,d$. 

Let us now argue by contradiction. Suppose there exists a point $z \in E$ with $\dist(z, L) \geq 1000 \epsilon r$. Then $z$ is linearly independent with respect to the family $y_1,...,y_d$ and therefore there cannot exist a $d$-plane $V$ so that $\dist(y_j, V) \leq C \epsilon r$ for all $j$ and $\dist(z, V) \leq C \epsilon r$. But this contradicts the fact that $\beta_{E, \infty}^d(x_0, cr) \lesssim_d \epsilon$ | which must hold from Lemma \ref{lemma:betap_betainfty}. The ensuing contradiction proves the lemma. 
\end{proof}

\begin{corollary} \label{corol:final2}
Let $E \subset \R^{n}$ be a lower content $d$-regular subset so that $E \subset B(0,1)$. Let $1 \leq p < \infty$. Except for a set of zero $\hd$ measure, if 
\begin{align*}
\int_0^{1} \betae{p}(B(x,t))^2 \frac{dt}{t} < \infty,
\end{align*}
then $x$ is a tangent point of $E$.
\end{corollary}
\begin{proof}
It is enough to show that for any $x \in E$,
\begin{align*}
    \sum_{k \in \N} \betae{p}(x, r_k)^2 \lesssim \int_0^{1} \betae{p}(x, t)^2\, \frac{dt}{t}. 
\end{align*}
Let $r_k \leq t \leq 10r_k$. Then recall from Lemma \ref{lemma:monotonicity}, that
$
\betae{p}(x, t)^2 \geq \left( \frac{r_k}{t} \right)^{2(d+p)} \betae{p}(x, r_k)^2. 
$
Thus
\begin{align*}
\int_{r_k}^{10r_k} \betae{p}(x,t)^2 \, \frac{dt}{t} \geq \int_{r_k}^{10r_k} \left( \frac{r_k}{t} \right)^{2(d+p)} \betae{p}(x, r_k)^2\, \frac{dt}{t} 
 \geq \frac{\ln(10)}{10^{2(d+p)}} \betae{p}(x, r_k)^2. 
\end{align*}
This let us conclude that
\begin{align*}
    \sum_{k \in \N} \betae{p}(x, r_k)^2 \lesssim_{p,d} \sum_{k \in \Z} \int_{r_k}^{10r_k} \betae{p}(x, t)^2 \, \frac{dt}{t}
     = \int_0^1 \betae{p}(x, t)^2 \, \frac{dt}{t}.
\end{align*}
The corollary then follows from Proposition \ref{theorem:summabletangent}.
\end{proof}

Given any bounded set $E$, we can translate and dilate it so that it is contained in the unit ball centered at $0$. Hence Corollary \ref{corol:final2} immediately gives one implication of Theorem \ref{theorem:main}.

\begin{remark}
One could think of an alternative way of proving this direction of Theorem \ref{theorem:main}; it was suggested to the author by Jonas Azzam. It goes as follows. Take a subset $E_1 \subset E$ where the Jones function is finite. First, reduce the problem to a subset $E_0 \subset E_1$ of positive and finite $d$-dimensional Hausdorff measure - as was done in Lemma \ref{lemma:reduction1}. Secondly, Frostmann's Lemma (e.g. see Theorem 8.8 in [M1]) guarantees the existence of a measure $\mu$ such that: $\mu$ is finite, it is supported on a compact subset of $E_0$ and it is upper $d$-regular, that is, $\mu(B(x, r))\leq r^d$ for any $x \in \R^n$ and $r>0$. Third, this implies that $\mu$ is bounded above by the $d$-dimensional Hausdorff content. Hence one can bound above the standard $\beta$ numbers with the content $\beta$ numbers. Fourth, this let us apply Theorem 1.1 in [AT1] and obtain rectifiability for $E_0$. Hence one can conclude as in Subsections 3.5 and 3.6.

Such a proof would be shorter and clearer It has, however, a drawback: as Theorem 1.1 in [AT1] holds only for $p=2$, one could obtain Theorem \ref{theorem:summabletangent} for $2 \leq p < \infty$. This is the only reason why we kept the current (longer and more technical) proof. 
\end{remark}

\section{\eqref{item:main_tan} implies \eqref{item:main_beta}}

Consider the decomposition of $E$ into the intrinsic cubes given in Theorem \ref{theorem:christ}. For a cube $Q \in \cubes$, set
\begin{align}
    \beta_E^{p,d}(Q) := \beta_E^{p,d}(B_Q), \label{eq:betacubes}
\end{align}
where 
$
B_Q= B(x_Q, C_1 \ell(Q)).
$

\begin{proposition} \label{theorem:tangentbeta}
Let $E \subset \R^{n}$ be lower content $d$-regular such that $E \subset B(0,1)$. If $d=1$ or $d=2$, let $1 \leq p < \infty$. If $d \geq 3$, let $1 \leq p < \frac{2d}{d-2}$. Then for $\hd$-almost all tangent points $x \in E$, it holds that 
\begin{align*}
    \sum_{\substack{Q \in \cubes; Q \subset B(0,1) \\ Q \ni x}} \beta_E^{p,d}(Q)^2 < \infty.
\end{align*}
\end{proposition}
The following lemma will be useful during the proof; it can be found in Mattila's book [M1], Chapter 15. 
\begin{lemma}[{[M1], Lemma 15.12}] \label{lemma:mattila}
Let $E \subset \R^n$, $V \in G(d, n)$, $\theta \in (0,\pi/2)$ and $0< r< \infty$. If 
\begin{align}\label{eq:cones}
    & E \setminus (X(a, V, \theta) \cap B(a, r)) = \emptyset \,\, \mbox{for all} \enskip a \in E, \mbox{ and }\\
    & \diam(E) < r,
\end{align}
then there exists a bi-Lipschitz map $f:V \supset \Pi_V(E) \to \R^n$ with $\Lip(f) = \frac{1}{\cos(\theta)}$ such that $E \subset f(V)$.
\end{lemma}

Recall the notation
\begin{align*}
& X(a, V, \theta) := \{ x \in \R^n| \dist(x-a, V) < \sin(\theta) |x-a| \} \\
& X(a, V^\perp, \theta) := \{ x \in \R^n| \dist(x-a, V^\perp) < \cos(\theta) |x-a| \},
\end{align*}
where $a \in \R^n$, $V \in G(d, n)$ and $V^\perp$ is the orthogonal complement of $V$.
From now on, let us write
\begin{align*}
    X(a, V^\perp, \theta, r) := X(a, V^\perp, \theta) \cap B(a, r).
\end{align*}

The following is an elementary fact, which we prove for completeness.
\begin{lemma}
Let $x \in E$ be a tangent point. Then it is a $C$-tangent point. 
\end{lemma}
\begin{proof}
Recall the definition of $C$-tangent in \ref{eq:tangent_cone}; if the lemma did not hold, for each $d$-dimensional affine plane $V$ we could find a $\theta \in (0,\pi/2)$ such that for all $r>0$, there existed a $y \in X(x, V^\perp, \theta, r) \cap E$; such a $y$ would have $\dist(y, V) \geq \sin(\theta)|y-x|$. One can easily see that this implies that $\limsup_{r \downarrow 0} \sup_{y \in B(x,r)\cap E} \dist(y,V)/r \sin(\theta)$.
\end{proof}

 Denote by $\dT(E)$ the set of tangent points of $E$. 
 We can assume without loss of generality that $\hd(\dT(E)) >0$, for otherwise there is nothing to prove. If $x \in \dT(E)$, denote its tangent by $L_x$. 
Fix $\theta \in (0, \pi/2)$ and let $\{L_k\}_{k\in \N} \subset $ be a dense countable subset of the Grassmannian,  $\{r_\ell\}_{\ell \in \N}$ be a dense, countable subset of $(0,1)$. Set
\begin{align*}
    X\left(x, L^\perp_k , \theta, r_\ell\right) := X \left(x, L^\perp_k , \theta\right) \cap B(x,r_\ell),
\end{align*}
and
\begin{align*}
    K_{n, \ell}:= \left\{x \in \dT(E) \, |\, X\left(x, L^\perp_k , \theta, r_\ell\right) \cap E = \emptyset \right\}. 
\end{align*}
Clearly,  
$
    \bigcup_{k,\ell \in \N} K_{n, \ell} = \dT(E).
$
Indeed, if $x \in \dT(E)$, then there exists an $r>0$ and an affine $d$-plane so that $X(x, L^\perp, \theta, r) \cap E = \emptyset$; since $\{L_k\}$ is dense, for each $\delta>0$, we may find an $k\in \N$ so that $\dist_H(L_{k}\cap B , (L_x-x) \cap B) < \delta$, for, say, $B=B(0,1)$. It is clear then, that we may find an $r_\ell \leq r$ so that 
$
    X\left(x, L^\perp_k, \theta, r_\ell\right) \cap E = \emptyset.
$
\begin{remark} \label{remark:lipK}
Without loss of generality, we can assume that $K_{k, \ell}$ is compact and that $\diam(K_{k, \ell}) \leq r_\ell/2$. By Lemma \ref{lemma:mattila}, each $K_{k, \ell}$ can be covered by the image of a bi-Lipschitz map $f: \Pi_{L_k}(K_{k, \ell}) \to \R^n$ with $\Lip(f) = \frac{1}{\cos(\theta)}$.  
\end{remark}

\begin{lemma} \label{lemma:betaontan}
Fix $k, \ell$ and consider $K_{k, \ell}$ as above. Let $p$ be given as in Proposition \ref{theorem:tangentbeta}. Then for $\hd$-almost every point $x \in K_{k,\ell}$, 
\begin{align}
    \sum_{\substack{Q \in \cubes, \, \ell(Q) \leq 1 \\ x \in Q}} \betae{p}(Q)^2 < \infty.
\end{align}
\end{lemma}
Denote by $Q_0$ the minimal cube in $\cubes$ such that $K \subset 3 Q_0$; we may assume that $\diam(Q_0) \leq \frac{1}{3}r_\ell$. Define
\begin{align}
& \mathcal{S}:= \left\{ Q \in \cubes; Q \subset B(0,1)\, |\, Q \cap K\neq \emptyset \enskip \mbox{ and } Q \subsetneq Q_0 \right\} \label{def:smallfam}\\
& \mathcal{L} := \left\{ Q \in \cubes; Q \subset B(0,1) \, |\, Q \cap K \neq \emptyset \enskip \mbox{ and } \enskip Q \supset Q_0 \right\}.
\end{align}

\begin{sublemma} \label{lemma:pointcubes}
We have
\begin{align}\label{eq:cubesest1}
\int_{K} \sum_{\substack{Q \in \cubes, \, \ell(Q) \leq 1 \\ x \in Q}} \beta_E^{p,d}(Q)^2 \, d \hd(x) \lesssim \sum_{Q \in \mathcal{S} \cup \mathcal{L}} \beta_E^{p,d} (Q)^2 \ell(Q)^d.
\end{align}
\end{sublemma}
\begin{proof}
We have that
\begin{align} \label{eq:cubesest3}
\int_{K} \sum_{\substack{Q \in \cubes, \, \ell(Q) \leq 1 \\ x \in Q}} \beta_E^{p,d} (Q)^2\, d \hd(x) & 
= \int_{K} \sum_{Q \in \mathcal{S} \cup \mathcal{L}} \beta_E^{p,d}(Q)^2 \chara_{\{(y, Q) \in K \times \cubes\, | \, y \in Q\}} (x)\, d \hd(x) \nonumber\\
& = \sum_{Q \in \mathcal{S} \cup \mathcal{L}} \beta_E^{p,d} (Q)^2 \int_{K} \chara_{\{y \in K\, |\, y\in Q\}} (x) d \hd(x) \nonumber\\
& \lesssim \sum_{Q \in \mathcal{S} \cup \mathcal{L}} \beta_{E}^{p,d}(Q)^2 \ell(Q)^d.
\end{align}
The second equality is an application of Fubini-Tonelli's theorem. For the inequality, notice that we are integrating over $K$. Such set is not necessarily Ahlfors regular, for the density may be zero in some balls. However the upper $d$-regularity is still maintained, since it is a subset of a $d$-dimensional Lipschitz graph. 
\end{proof}
\begin{sublemma} \label{lemma:largecubes}
\begin{align} \label{eq:bigcubessum}
    \sum_{Q \in \mathcal{L}} \beta_E^{p,d} (Q)^2 \ell(Q)^d <\infty.
\end{align}
\end{sublemma}
\begin{proof}
Clearly,
$
\betae{p}(Q)^2 \lesssim 1.
$
Also recall that any $Q \in \mathcal{L}$ is contained in $B(0,1)$. Moreover, for each generation, there can be at most one cube which contains $Q_0$. Hence the sum \eqref{eq:bigcubessum} is really a finite sum of bounded terms, and therefore must be finite.
\end{proof}
Let us now deal with the sum over the family $\mathcal{S}$. 
Without loss of generality, take $L_k$ (recall that $K=K_{k,\ell}$) to be $\R^d$. 
We extend $f$ (as given in Remark \ref{remark:lipK}) to the whole $\R^d$; if we call $g$ such an extension, it is a standard fact that $\Lip(g) \sim_{n,d} \Lip(f)$. Set 
\begin{align}
    G := g(\R^d).\nonumber
\end{align}
We apply Lemma \ref{lemma:azzamschul}, so to obtain
\begin{align} \label{eq:Error}
    \betae{p}(Q)^2 \lesssim \beta_{G}^{p,d}(Q)^2 + \ps{\frac{1}{\ell(Q)^d}\int_{E \cap 2B_Q} \ps{\frac{\dist(y, G)}{\ell(Q)}}^p \, d \hdc(y) }^{\frac{2}{p}} =: A+B. 
\end{align}
 Note that, because $G$ is the Lipschitz image of a plane, then
 \begin{align} \label{eq:smallLip}
     A=\sum_{Q \in \mathcal{S}} \beta_{G}^{p,d} (Q)^2 \ell(Q)^d <\infty
 \end{align}
 by Theorem I in [AS1]. 

Before estimating $B$ in \eqref{eq:Error}, we will construct a `Whitney decomposition' for $Q_0 \setminus K$.
We proceed as follow. Because $\Pi_{\R^d}(K)$ is compact, there exists a Whitney decomposition of $\Pi_{\R^d}(K)^c$ in $\R^d$; let us denote such a decomposition by $\cW_d$. For $S \in \cW_d$, set
\begin{align}
    T_S:= \ps{S \times \R^{n-d}} \cap Q_0,
\end{align}
and
\begin{align}
    \cW_n := \{ T_S \neq \emptyset \, |\, S \in \cW_d\}.
\end{align}
\begin{sublemma}
For $T_S \in \cW_n$, 
\begin{align}
\diam(T_S) \lesssim_{\Lip(f)} \diam(S).
\end{align}
\end{sublemma}
\begin{proof}
Let $y,z \in T_S \subset Q_0 \setminus K$. For brevity, let us write $\Pi:= \Pi_{\R^d}$ and $ \widetilde K := \Pi(K)$.
Pick $p=p(y) \in \wt K$ such that 
\begin{align} \label{eq:choice2}
    |p - \Pi(y)| \leq 2 \dist(\Pi(y), \wt K),
\end{align}
and set 
\begin{align} \label{eq:choice1}
    x(y) := f(p) \in K.
\end{align}
We choose $q=q(z) \in \wt K$ and $x(z)= f(q) \in K$ is the same manner. 
Now, let $L_y$ (resp. $L_z$) be the $d$-plane parallel to $\R^d$ which contains $x(y)$ (resp. $x(z)$). Denote by $\Pi_y$ (resp. $\Pi_z$) the orthogonal projection onto $L_y$ (resp. $L_z$). 
Then set
\begin{align}
   & \wt y := \Pi_y (y) \label{eq:choice4} \\
   &  \wt z := \Pi_z (z).
\end{align}
Then,
$
    |y - z| \leq |y - \wt y| + |\wt y - \wt z| + |\wt z- z|. 
$
We see that
\begin{align*}
    | y - \wt y|  = |(y-x(y)) - (\wt y - x(y))| 
     = \dist(y - x(y), \R^d).
\end{align*}
Since $x(y) \in K$ and $y \in Q_0$ (and thus in particular $|x(y)-y| < r_\ell$), then 
\begin{align} \label{eq:dist1}
    \dist(y- x(y), \R^d) \leq \sin(\theta) |y-x(y)|,
\end{align}
and also
\begin{align} \label{eq:dist2}
    \dist(y - x(y), (\R^d)^\perp)  \geq \cos(\theta) |y - x(y)|. 
\end{align}
Furthermore, notice that 
\begin{align} \label{eq:dist3}
    \dist( y-  x(y), (\R^d)^\perp) = |\Pi(x(y)) - \Pi(y)|.
\end{align}
Thus, \eqref{eq:dist1}, \eqref{eq:dist2}, \eqref{eq:dist3} and the choice of $x(y)$ (as in \eqref{eq:choice2} and \eqref{eq:choice1}), give
\begin{align*}
    |y- \wt y| & \leq \sin(\theta) |x(y)-y| \\
    & \leq \frac{\sin(\theta)}{\cos(\theta)} \dist(y - x(y), (\R^d)^\perp) \\
    & = \frac{\sin(\theta)}{\cos(\theta)} |\Pi(x(y))- \Pi(y)| \\
    & \leq \frac{1}{\cos(\theta)} |p - \Pi(y) | \\
    & \leq \Lip(F) 2 \dist(\Pi(K), \Pi(y)).
\end{align*}
Since $\Pi(y) \in S$, and $S$ is a Whitney cube, we see that 
\begin{align*}
    \dist(\wt K, \Pi(y)) \leq \dist(\wt K, S) + \diam (S)  \lesssim \diam(S).
\end{align*}
We can conclude that
\begin{align} \label{eq:dist4}
    |y - \wt y| \lesssim_{\Lip(f)} \diam(S).
\end{align}
The same argument gives $|z- \wt z| \lesssim_{\Lip(f)} \diam(S)$. 

We need to estimate $|\wt y - \wt z|$. We have
\begin{align*}
    |\wt y- \wt z| \leq |\wt y - g(\Pi(\wt y))|+ |g(\Pi(\wt y)) - g(\Pi(\wt z))| + |g(\Pi(\wt z))- \wt z|. 
\end{align*}
Now, 
\begin{align} \label{eq:lemmawhit2}
    |g(\Pi(\wt y)) - g(\Pi(\wt z))| \leq \lip(f) |\Pi(\wt y)- \Pi(\wt z)| \leq \Lip(f) \diam(S),
\end{align}
since $\Pi(\wt y)$, $\Pi(\wt z) \in S$.
On the other hand, again using the choice of $x(y)$,  
\begin{align} \label{eq:lemmawhit3}
    |\wt y - g(\Pi(\wt y))|  &\leq |\wt y - x(y)| + |x(y) - g(\Pi(\wt y))| \nonumber\\
    & = |\Pi(x(y))- \Pi(y)| + |g(p) - g(\Pi(y))| \nonumber\\
    & = |p - \Pi(y)| + |g(p) - g(\Pi(y))| \nonumber\\
    & \lesssim \Lip(f) \diam(S). 
\end{align}
This, together with \eqref{eq:lemmawhit2} and \eqref{eq:dist4}, give the lemma. 
\end{proof}

\begin{sublemma} For $T_S \in \cW_n$,
\begin{align} \label{eq:whitney_diams2}
\dist(T_S, K) \sim \diam(S).
\end{align}
\end{sublemma}
\begin{proof}
Since $\Pi$ is $1$-Lipschitz and $S$ is a Whitney cube, 
\begin{align*}
\dist(T_S, K) \geq \dist(S, \Pi(K)) \sim \diam(S).
\end{align*}

On the other hand, if we let $y \in T_S$, $x(y)$ as in \eqref{eq:choice2} and \eqref{eq:choice1}, we see that
\begin{align}\label{eq:dist5}
    \dist(T_S, K) & \leq |y - x(y)| \nonumber\\
    & \leq \frac{1}{cos(\theta)} \dist(y-x(y), (\R^d)^\perp) \nonumber\\
    & = \frac{1}{\cos(\theta)} |\Pi(x(y))-\Pi(y)| \nonumber\\
    & = \frac{1}{\cos(\theta)} |p- \Pi(y)|\nonumber \\
    & \lesssim_{\Lip(f)} \diam(S).
\end{align}
\end{proof}

Furthermore, note that
\begin{align} \label{eq:Tswhitney0}
    Q_0 \setminus K \subset \bigcup_{S \in \cW_d} T_S, 
\end{align}
and that, if $P, S \in \cW_d$ are so that $P \cap S = \emptyset$, then 
\begin{align} \label{eq:Tswhitney1}
    T_S \cap T_P = \emptyset.
\end{align}
Thus $\cW_n$ is a disjoint decomposition of $Q_0 \setminus K$ so that
\begin{align} \label{eq:whitney_diams}
     \hdc(T_S)\lesssim_{\Lip(f)} \hdc(S).
\end{align}

We state one last lemma before proving the summability of the error term. 
\begin{sublemma} \label{lemma:distF}
Let $y \in T_S \in \cW_n$. Then
\begin{align*}
    \dist(y, G) \lesssim_{\Lip(f)} \diam(S).
\end{align*}
\end{sublemma}
\begin{proof}
For $y \in T_S$, let $x(y)$ and $p=p(y)$ as above. Then
$
    \dist(y, G)   \leq |y - g(p)| 
     = |y - x(y)|.
$
If we now argue as in \eqref{eq:dist5}, the lemma follows.
\end{proof}

Let us now go back to the proof of Lemma \ref{lemma:betaontan}. Recall that we had to give an estimate for the error term $B$ in \eqref{eq:Error}.
\begin{sublemma} \label{lemma:errorBbound}
We have 
\begin{align}
    \sum_{Q \in \mathcal{S}} \ps{\frac{1}{\ell(Q)^d}\int_{E \cap 2B_Q} \ps{\frac{\dist(y, G)}{\ell(Q)}}^p \, d \hdc(y) }^{\frac{2}{p}} < \infty.
\end{align}
\end{sublemma}
\begin{proof}
Without loss of generality, we may assume that $p \geq 2$, since, by Lemma \ref{lemma:betaincreasep}, $\beta_E^{p, d} \lesssim \beta_E^{2,d}$ if $p \leq 2$.
Let $Q \in \mathcal{S}$; then we have that
\begin{align}
 \int_{2B_Q \cap E} \ps{\frac{\dist(y, G)}{\ell(Q)}}^p \,  d\hdc(y) 
    & \leq  \int_{(Q_0 \setminus K) \cap 2B_Q} \ps{\frac{\dist(y, G)}{\ell(Q)}}^p \,  d\hdc(y),
\end{align}
where we used the fact that the integrand is nonzero only if $y \in Q_0 \setminus K$, since $K \subset G$.
Pick a constant $c_1 \geq 1$ so that if $Q' \in \cubes$ is the parent cube of $Q$, then $2B_Q \cap E \subset c_1 Q \subset 2B_{Q'}$. This may be done if one set $\delta>0$ in Theorem \ref{theorem:christ} small enough. Using \eqref{eq:Tswhitney0} and Lemma \ref{lemma:distF}, we see that
\begin{align*}
\int_{\ps{Q_0 \setminus K}\cap 2B_Q} \ps{\frac{\dist(y, G)}{\ell(Q)}}^p \,  d\hdc(y) & \leq \sum_{\substack{T_S \in \cW_{n} \\ T_S \cap c_1 Q \neq \emptyset}} \int_{T_S} \ps{\frac{\dist(y, G)}{\ell(Q)}}^p \, d \hdc (y)\\
& \lesssim \sum_{\substack{T_S \in \cW_{n}\\ T_S \cap c_1 Q \neq \emptyset}} \frac{\hdc(T_S) \diam(T_S)^p}{\ell(Q)^p}
\end{align*}
Thus, recalling the well known fact that if $0<\alpha\leq1$ and $a,b$ are two nonnegative real numbers, then $(a+b)^\alpha \leq a^\alpha + b^\alpha$, we get that
\begin{align}
    & \sum_{Q \in \mathcal{S}} 
    \ps{ 
    \sum_{\substack{T_S \in \cW_{n}\\ T_S \cap c_1 Q \neq \emptyset}} \frac{\hdc(T_S) \diam(T_S)^p}{\ell(Q)^p} \ell(Q)^{d\ps{\frac{p-2}{2}}}
    }^{\frac{2}{p}} \nonumber\\
     & \lesssim \sum_{Q \in \mathcal{S}} 
     \sum_{\substack{T_S \in \cW_{n}\\ T_S \cap c_1 Q \neq \emptyset}} \left( \frac{\diam(T_S)^{p+d}}{\ell(Q)^{p+d}}\right)^{\frac{2}{p}} \, \ell(Q)^d 
    \lesssim \sum_{T_S \in \cW_n}  \sum_{\substack{Q \in \mathcal{S} \\ c_1 Q \cap T_S \neq \emptyset}} \frac{\diam(T_S)^{\frac{2d}{p} + 2}}{\ell(Q)^{\frac{2d}{p} - d +2}} \label{eq:whitney2}
\end{align}
Notice that, 
\begin{align*}
    \mbox{if } d=1 \mbox{ or } d=2\, \mbox{ then }\enskip  2-\frac{2}{p}d\ps{\frac{p-2}{2}} > 0 \enskip \forall p \geq 2.
\end{align*}
On the other hand,
\begin{align*}
    \mbox{if } d \geq 3 \mbox{ then } \enskip 2-\frac{2}{p}d\ps{\frac{p-2}{2}} > 0  \mbox{ for } 2 \leq p < \frac{2d}{d-2}.
\end{align*}
Given the hypotheses in Proposition \ref{theorem:tangentbeta}, in either case we may set
\begin{align*}
    2-\frac{2}{p}d\ps{\frac{p-2}{2}} =: \alpha > 0.
\end{align*}

Now let $z \in T_S \cap Q$. Then we see, using \eqref{eq:whitney_diams2} and recalling that, by definition of $\mathcal{S}$, $K \cap Q\neq \emptyset$,
\begin{align*}
\diam(T_S) \sim \dist(T_S, K)  \leq \dist(T_S, z)+ \dist(z, K)
\lesssim \ell(Q).
\end{align*}
This implies that, given a Whitney cube $T_S \in \cW_n$, and given $k \in \Z$, the number of cubes belonging to $\mathcal{D}_k$ (i.e. the number of cubes of $k$-th generation) such that $T_S \cap c_1 Q \neq \emptyset$ is bounded above by an universal constant.
Hence the interior sum in \eqref{eq:whitney2} is a geometric series. Thus we have
\begin{align*}
   \sum_{\substack{Q \in \mathcal{S} \\ c_1 Q \cap T_S \neq \emptyset}} \frac{\diam(T_S)^{\frac{2d}{p} + 2}}{\ell(Q)^{\frac{2d}{p} - d +2}} & \lesssim\sum_{\substack{Q \in \mathcal{S} \\ c_1 Q \cap T_S \neq \emptyset}} \ell(Q)^d
   \lesssim \diam(T_S)^d.
\end{align*}
Hence we obtain
\begin{align*}
    \eqref{eq:whitney2} & \lesssim \sum_{T_S \in \cW_n} \hdc(T_S)
    \lesssim \sum_{S \in \cW_d} \hdc(S) 
     \lesssim r_\ell^d.
\end{align*}
This proves the Sublemma.
\end{proof}
We now see that \eqref{eq:cubesest3} together with Sublemma \ref{lemma:largecubes}, Sublemma \ref{lemma:errorBbound} and \eqref{eq:smallLip} give
\begin{align*}
\int_{K} \sum_{\substack{Q \in \cubes, \, \ell(Q) \leq 1 \\ x \in Q}} \betae{2}(Q)^2 \, d \hdc(y) < \infty
\end{align*}
and therefore
\begin{align*}
    \sum_{\substack{Q \in \cubes, \, \ell(Q) \leq 1 \\ x \in Q}} \beta_E^{2,d}(Q)^2 < \infty \enskip \enskip \mbox{for} \enskip \enskip \hd \mbox{-almost all} \enskip x \in K. 
\end{align*}
This concludes the proof of Lemma \ref{lemma:betaontan} and therefore of Proposition \ref{theorem:tangentbeta}.

\begin{corollary} \label{corol:final}
Let $E \subset \R^{n}$ be a $d$-lower content regular set such that $E \subset B(0,1)$. If $d=1$ or $d=2$, let $1 \leq p < \infty$. If $d \geq 3$, let $1 \leq p < \frac{2d}{d-2}$. Then, except for a set of zero $\hd$ measure, if $x \in E$ is a tangent point of $E$ then
\begin{align*}
\int_0^1 \beta_E^{p,d} (x,t)^2 \, \frac{dt}{t} < \infty.
\end{align*}
\end{corollary}
\begin{proof}
Let $x \in E$ be a tangent point of $E$.
Recall Lemma \ref{lemma:monotonicity}. Then, for $1\leq p < \infty$,
\begin{align}
    \int_0^1 \beta_E^{p,d} (x,t)^2 \, \frac{dt}{t} & \lesssim \int_0^1 \beta_E^{p,d}(Q)^2 \, \frac{dt}{t} \nonumber\\
    & \sim \sum_{Q \ni x} \beta_E^{p,d}(Q)^2 \, \int_0^1 \chara_{\{t \in [0,\infty)| C^{-1}\ell(Q) \leq t \leq C \ell(Q)\}} \, \frac{dt}{t} \nonumber\\
    & = \sum_{Q \ni x} \beta_E^{p,d}(Q)^2 \, \int_{C^{-1}\ell(Q)}^{C \ell(Q)} \, dt/t \nonumber\\
    & = \ln (C^2) \sum_{Q \ni x} \beta_E^{p,d}(Q)^2.  \label{eq:corollary}
\end{align}
This sum is bounded by Theorem \eqref{theorem:tangentbeta}.
\end{proof}

Now, any bounded set $E \subset \R^n$ can be dilated and translated so that it is contained in $B(0,1)$. Thus Theorem \ref{theorem:main} follows immediately from Corollaries \ref{corol:final2} and \ref{corol:final}.

\input{06_biblio}
\Addresses


\bibliographystyle{amsplain}

\end{document}